\documentclass[12pt]{article}
\usepackage{amsmath,amsthm,amsfonts,amssymb,amscd, amsxtra}
\usepackage[latin1]{inputenc}
\newtheorem{theorem}{Theorem}
\newtheorem{lemma}[theorem]{Lemma}

\newtheorem{corollary}[theorem]{Corollary}
\newtheorem{proposition}[theorem]{Proposition}

\newtheorem{remark}{Remark}

\begin{document}
%%%%%%%%%%%%%%%%%%%%%%%%%%%%%%%%%%%%%%%%%%%%%%%%%%%%%%%%%%%%%
\title{Convergence of the Gauss-Newton method for a special class of 
systems of equations under a majorant condition}

\author{  M. L. N. Gon\c calves \thanks{IME/UFG, Campus II - Caixa
    Postal 131, CEP 74001-970 - Goi\^ania, GO, Brazil (E-mail:{\tt
      maxlng@mat.ufg.br}). The author was supported in part by CAPES and
    CNPq Grant 473756/2009-9.} \and P. R. Oliveira \thanks{COPPE-Sistemas, Universidade Federal do Rio de Janeiro,
Rio de Janeiro, RJ 21945-970, Brazil (Email: {\tt poliveir@cos.ufrj.br}).
This author was supported in part by CNPq.} }

\maketitle
%%%%%%%%%%%%%%%%%%%%%%%%%%
\begin{abstract}

In this paper, we study the Gauss-Newton method for a special class of  systems of nonlinear equation. Under the hypothesis
that the derivative of the function under consideration satisfies a majorant condition, semi-local convergence analysis
 is presented. In this analysis the conditions and proof of convergence are simplified by using a simple majorant condition to define regions where the Gauss-Newton sequence is ``well behaved". Moreover, special cases of the general theory are presented as applications.
%In this analysis the conditions and proof of convergence are simplified by using a majorant condition and define regions where sequence of Gauss-Newton is ``well behaved".
%Moreover, the auxiliary function does not have to be defined beyond its first root.

\end{abstract}
%\noindent  \textsc{AMSC: 47J15, 65H10.}


\maketitle

%%%%%%%%%%%%%%%%%%%%%%%%%%%%%%%%%%%%%%%%%%%%%%%%%%%%%%%%
\section{Introduction}\label{sec:int}
%%%%%%%%%%%%%%%%%%%%%%%%%%%%%%%%%%%%%%%%%%%%%%%%%%%%%%%%
Consider the {\it  systems of nonlinear equations}
\begin{equation} \label{eq:11}
F(x)=0,
\end{equation}
where $F:\Omega\to \mathbb{R}^m$ is a continuously differentiable function and $\Omega\subseteq\mathbb{R}^n$ is an open set.

When $F '(x)$ is invertible, %i.e., the nonlinear systems of equations in \eqref{eq:11} is square,
the Newton method and its variant (see \cite{F08,FS04,MAX1,FS06}) are the most efficient methods known for solving \eqref{eq:11}. However, when $F '(x)$ is not
necessarily invertible, a generalized Newton method, called the Gauss-Newton method (see \cite{MAX3,MAX4,MAX2}), defined by
$$
x_{k+1}={x_k}- F'(x_k)^{\dagger}F(x_k),
\qquad k=0,1,\ldots,
$$
where $F'(x_k)^{\dagger}$ denotes the Moore-Penrose inverse of the linear operator $F'(x_k)$,
finds least squares solutions of \eqref{eq:11} which may or may not be solutions of \eqref{eq:11}.
These least squares solutions are related
to the nonlinear least squares problem
\begin{equation}\label{eq:p1}
\min_{x\in \Omega } \;\|F(x)\|^2,
\end{equation}
that is, they are stationary points of $G(x)=\|F(x)\|^2.$
It is worth noting that, if $F'(x)$ is surjective, then least squares solutions of systems of nonlinear equations are also solutions of systems of nonlinear equations.

We shall consider the same special class of  systems of nonlinear equations studied in \cite{MR817125,MR2437700,Li2010}, i.e., systems of nonlinear equations where the function $F$ under consideration satisfies
 \begin{equation} \label{rank1}
\left\|F'(y)^{\dagger}(I_{\mathbb{R}^m}-F'(x)F'(x)^{\dagger})F(x)\right\| \leq \kappa \|x-y\|, \qquad \forall \;x,y \;\in \Omega
\end{equation}
 for some $0 \leq \kappa < 1$ and $I_{\mathbb{R}^m}$ denotes the identity operator on $\mathbb{R}^m$.
 This special class of nonlinear systems of equation contains underdertermined systems with surjective derivatives, because when $F'(x)$ is surjective we can prove that
 $k=0$ in \eqref{rank1}.

 In recent years, papers have addressed the issue of convergence of the Newton method, including the Gauss-Newton method, by relaxing the assumption of Lipschitz continuity of the derivative (see \cite{Dedieu2002,MR16517501,MAX3,F08,FS04,MAX1,MAX4,MAX2,FS06,MR2437700,Li2010,XW10,MR2359960,MR2429168} and references therein). These new assumptions also allow us to unify previously unrelated convergence results, namely results for analytical functions ($\alpha$-theory or $\gamma$-theory) and the classical results for functions with Lipschitz derivative. The main new conditions that relax the condition of Lipschitz continuity of the derivative include the majorant condition, which we will use, and Wang's condition, introduced
in \cite{XW10} and used for example in \cite{Li2010,MR2359960,MR2429168} to study the Gauss-Newton method. In fact, under the hypothesis in this paper, it can be shown that these conditions are equivalent. However, the formulation as a majorant condition is in a sense better than Wang's condition, as it provides a clear relationship between the majorant function and the nonlinear function under consideration. Besides, the majorant condition provides a simpler proof of convergence.

Following the ideas of the semi-local convergence analysis in \cite{MAX4,FS06}, we will present a new semi-local convergence analysis of the Gauss-Newton method for solving \eqref{eq:11}, where $F$ satisfies \eqref{rank1}, under a majorant condition. The convergence analysis presented here communicates the conditions and proof in a quite simple manner. This is possible thanks to our majorant condition and to a demonstration technique introduced in~\cite{FS06} which, instead of looking only to the sequence generated, identifies regions where, for the problem under consideration, the Gauss-Newton sequence is well behaved, as compared with a method applied to an auxiliary function associated with the majorant function.
Moreover, two unrelated previous results relating to the Gauss-Newton method are unified, namely, results for analytical functions under an $\alpha$-condition and the classical result for functions with Lipschitz derivative. Besides, convergence results for underdetermined systems with surjective derivatives will be also given.

 The paper is organized as follows. Sect.
\ref{sec:int.1} lists some notations and basic results used in
the presentation. Sect. \ref{lkant} states and proves the main results. Finally, special cases of the general theory are presented as applications
in Sect. \ref{sec:ec}.

%%%%%%%%%%%%%%%%%%%%%%%%%%%%%%%%%%%%%%%%%%
\subsection{Notation and auxiliary results} \label{sec:int.1}
%%%%%%%%%%%%%%%%%%%%%%%%%%%%%%%%%%%%%%%%%
The following notations and results are used throughout this
presentation. Let $\mathbb{R}^n$ be with a norm $\|.\|$.  The open and closed balls
at $a \in \mathbb{R}^n$ and radius $\delta>0$ are denoted, respectively by
$$
B(a,\delta) :=\{ x\in \mathbb{R}^n ;\; \|x-a\|<\delta \}, \qquad B[a,\delta] :=\{ x\in \mathbb{R}^n ;\; \|x-a\|\leqslant \delta \}.
$$

Given a linear operator $A: \mathbb{R}^n \to \mathbb{R}^m$ (or an $n\times m$ matrix), the Moore-Penrose inverse of $A$
 is the linear operator $A^\dagger:\mathbb{R}^m \to \mathbb{R}^n$ (or an $m\times n$ matrix) which satisfies:
$$AA^{\dagger}A=A, \quad A^{\dagger}AA^{\dagger}=A^{\dagger}, \quad (AA^{\dagger})^{*}=AA^{\dagger}, \quad (A^{\dagger}A)^{*}=A^{\dagger}A, $$
where $A^*$ denotes the adjoint of $A$.
The Kernel and image of A are denoted by $Ker(A)$ and $im(A)$, respectively. It is easily seen
from the definition of the Moore-Penrose inverse that
\begin{equation}\label{1111}
AA^{\dagger}= \Pi_{Ker(A)^{\bot}}, \qquad A^{\dagger}A= \Pi_{im(A)},
\end{equation}
 where $\Pi_{E}$ denotes the projection of $\mathbb{R}^n$ onto subspace E.

We use  $\mbox{I}_{\mathbb{R}^m}$ to denote the identity operator on ${\mathbb{R}^m}$. If $A$ is surjective, then
 \begin{equation}\label{121212}
A^{\dagger}=A^{*}(AA^*)^{-1}, \qquad AA^{\dagger}=I_{\mathbb{R}^m}, \qquad (AA^{\dagger})^{\dagger}=AA^{\dagger}.
\end{equation}

%%%%%%%%%%%%%%%%%%%%%%%%%%
\begin{lemma}(Banach's Lemma) \label{lem:ban1}
Let $B:\mathbb{R}^n \to \mathbb{R}^n$ be a continuous linear operator.
If $\|B-I_{\mathbb{R}^n}\|<1$, then $B$ is invertible and $ \|B^{-1}\|\leq
1/\left(1- \|B-I_{\mathbb{R}^n}\|\right). $
\end{lemma}
\begin{proof} See the proof of Lemma 1, p.189 of Smale \cite{S86} with $A=I_{\mathbb{R}^n}$ and $c=\|B-I_x\|$.
\end{proof}
The next lemma is proved on p.43 of \cite{LAW1} (see also \cite{ISRAEL2003}).  It is on the perturbation of the Moore-Penrose inverse of $A$.

\begin{lemma} \label{lem:ban}
Let $A, B: \mathbb{R}^n \to \mathbb{R}^m$ be continuous linear operators. Assume that
$$1 \leq rank(B)\leq rank(A), \qquad \|A^\dagger\|\|A-B\|<1.$$
Then
$$
rank(A)= rank(B), \qquad \|A^\dagger\|\leq \frac{\|B^\dagger\|}{ 1- \|B^\dagger\|\|A-B\|}.
$$
\end{lemma}

%%%%%%%%%%%%%%%%%%%%%%%%%
%%%%%%%%%%%%%%%%%%%%%%%%%%%%%%%%%%%%%%%%%%%%%%%%%%%%%
\section{Semi-local analysis for the Gauss-Newton method} \label{lkant}
%%%%%%%%%%%%%%%%%%%%%%%%%%%%%%%%%%%%%%%%%%%%%%%%%%%%%
Our goal is to state and prove a semi-local theorem of the Gauss-Newton method for solving nonlinear systems of equations, where
the function under consideration satisfies \eqref{rank1}. First, we will prove that this theorem holds for an auxiliary function
associated with the majorant function. Then, we will prove well-definedness of the Gauss-Newton method and convergence. Convergence rates
will also be established. The statement of the theorem is:
%%%%%%%%%%%%%%%%%%%%
\begin{theorem}\label{th:knt}
Let $\Omega\subseteq\mathbb{R}^n$ be an open set and $F:\Omega\to \mathbb{R}^m$ a continuously differentiable function.
 Suppose that
 \begin{equation} \label{rank}
\left\|F'(y)^{\dagger}(I_{\mathbb{R}^m}-F'(x)F'(x)^{\dagger})F(x)\right\| \leq \kappa \|x-y\|,\qquad \forall \;x,y \;\in \Omega
\end{equation}
for some $0 \leq \kappa < 1$. Take $x_0\in\Omega$ such that $\beta:=\|F'(x_0)^{\dagger}F(x_0)\|>0$, $F'(x_0)\neq 0$ and
 \begin{equation} \label{posto}
 rank(F'(x))\leq rank(F'(x_0)), \qquad \forall \; x \in \Omega.
\end{equation}
Suppose that there exist $R>0$ and a continuously differentiable function $f:[0,\; R)\to \mathbb{R}$ such that,
$B(x_0,R)\subseteq \Omega$,
 \begin{equation}\label{KH}
 \|F'(x_0)^{\dagger}\| \|F'(y)-F'(x)\| \leq f'(\|y-x\|+\|x-x_0\|)-f'(\|x-x_0\|),
 \end{equation}
  for any $x,y\in B(x_0, R)$, $\|x-x_0\|+\|y-x\|< R$,
     \begin{itemize}
  \item[{\bf h1)}] $f(0)=0$, $f'(0)=-1$;
  \item[{\bf h2)}] $f'$ is convex and strictly increasing.
  \end{itemize}
 Take $\lambda\geq 0$ such that $\lambda\geq -\kappa f'(\beta)$ and consider the auxiliary function $h_{ \beta, \lambda}:[0,\; R)\to \mathbb{R}$,
 \begin{equation}\label{defh}
h_{ \beta, \lambda}(t):= \beta + \lambda  t+ f(t).
\end{equation}
If $h_{ \beta, \lambda}$ satisfies
\begin{itemize}
  \item[{\bf h3)}]   $h_{ \beta, \lambda}(t)=0$   for some  $t \in (0,  R)$,
\end{itemize}
then $h_{ \beta, \lambda}(t)$ has a smallest zero   $t_* \in (0,  R)$,  the  sequences
for solving  $ h_{ \beta, \lambda}(t)=0$ and $F(x)=~0$, with starting point $t_0=0$ and $x_0$,  respectively,
\begin{equation}\label{ns.KT}
    t_{k+1} ={t_{k}}-h_{ \beta,0}'(t_{k}) ^{-1}h_{ \beta, \lambda}(t_{k}),\quad x_{k+1}={x_k}-F'(x_k)^{\dagger}F(x_k), \quad k=0,1,\ldots\,,
\end{equation}
are well defined,  $\{t_{k}\}$ is strictly increasing, is contained in
  $[0,t_*)$, and converges  to $t_*$, $\{x_k\}$ is contained in $B(x_0,t_*)$, converges to a point  $x_*\in B[x_0, t_*]$ such that
$F'(x_*)^{\dagger}F(x_*)=0$  and
\begin{equation}\label{eq:bd}
  \|x_{k+1}-x_{k}\|   \leq  t_{k+1}-t_{k} , \qquad \|x_*-x_{k}\|   \leq  t_*-t_{k},\qquad  k=0, 1, \ldots\, ,
  \end{equation}
 \begin{equation}\label{eq:002}
\|x_{k+1}-x_{k}\|\leq   \frac{t_{k+1}-t_{k}}{(t_{k}-t_{k-1})^2} \|x_k-x_{k-1}\|^2,
 \quad k=1,2, \ldots\,.
\end{equation}
Moreover, if $\lambda=0$ ($\lambda=0$ and  ${h_{ \beta, 0}'}(t_*)<0$), the sequences $\{t_k\}$ and $\{x_k\}$ converge Q-linearly and R-linearly
 (Q-quadratically and R-quadratically) to $t_{*}$ and  $x_{*}$, respectively.
\end{theorem}

\begin{remark}\label{lambda}
It is easily seen that the best choice of  $\lambda$ is the smallest  possible. Hence,  if  $f'(\beta)\leq 0$ then $\lambda=- \kappa f'(\beta)$ is the best choice. Moreover, since $-f'(\beta)<-f'(0)=1$ $({\bf h2})$, a possible choice for
$\lambda$ is $\kappa$, despite not being the best.
\end{remark}

\begin{remark}
If $F'(x)$ is surjective, it follows from the second equation in \eqref{121212} that $F'(x)F'(x)^{\dagger}=I_{\mathbb{R}^m}$. Thus, we can take $\lambda=0$, because $F$ satisfies \eqref{rank} with $\kappa=0$.  Therefore,
in this case, Theorem~\ref{th:knt}
extends the results obtained by Ferreira and Svaiter in Theorem 2 of \cite{FS06}.
\end{remark}

From now on, we assume that the hypotheses of Theorem \ref{th:knt} hold.
%%%%%%%%%%%%%%%%%%%%%%%%%%%%%
\subsection{ The  auxiliary function and sequence $\{t_k\}$ } \label{mc}
%%%%%%%%%%%%%%%%%%%%%%%%%%%%%%%%%%%%%%%%%%%%%%%%%%%%%
In this section, we will study the auxiliary function, $h_{ \beta, \lambda}$, which is associated with  the majorant function, $f$, and prove all results regarding only the sequence $\{t_k\}$. Remember that a function that satisfies \eqref{KH}, {\bf h1} and {\bf h2} is called a majorant function for the function F on $B(x_0,R)$. More details about the majorant condition can be found in \cite{MAX3,F08,FS04,MAX1,MAX4,MAX2,FS06}.

%%%%%%%%%%%%%%%
\begin{proposition} \label{pr:10} The following statements  hold:
\begin{itemize}
 \item[{\bf i)}]  $h_{ \beta, \lambda}(0)=\beta >0$,   $h'_{ \beta, \lambda}(0)=\lambda-1$;
  \item[{\bf  ii)}]  $h'_{ \beta, \lambda}$ is convex and strictly increasing.
   \end{itemize}
\end{proposition}
\begin{proof}
It follows from the definition in \eqref{defh} and assumptions  {\bf h1} and {\bf h2}.
 %To establish the  second part, first note that $h_{ \beta, \lambda}$ is  strictly convex, since $h_{ \beta, \lambda}'$ is strictly increasing.
%Hence, using  $h_{ \beta, \lambda}(0)=\beta >0$ and {\bf h3}, we obtain  $h'_{ \beta, \lambda}(0)=\lambda-1<0,$ that is, $\lambda<1$.
\end{proof}
\begin{proposition} \label{pr:1}
 The function $h_{ \beta, \lambda}$ has a smallest root $t_*\in (0,R) $,  is strictly convex, and
  \begin{equation} \label{eq:n.f}
    h_{ \beta, \lambda}(t)>0, \quad h_{ \beta, 0}'(t)<0, \qquad t<t-h_{ \beta, \lambda}(t)/h_{ \beta, 0}'(t)< t_*,
    \quad\qquad  \forall ~t\in [0,t_*).
  \end{equation}
  Moreover,  $h_{ \beta, 0}'(t_*)\leq 0$.
  %\end{equation}
\end{proposition}
\begin{proof}
As $h_{ \beta, \lambda}$ is continuous in $[0,R)$ and have a zero there ({\bf h3}), it must have a smallest zero $t_*$,
which is greater than $0$ because $h_{ \beta, \lambda}(0)=\beta>0.$  Since $h_{ \beta, \lambda}'$ is strictly increasing by item {\bf ii} of Proposition \ref{pr:10}, $h_{ \beta, \lambda}$ is  strictly convex.

 The first inequality in \eqref{eq:n.f} follows from the assumption
  $h_{ \beta, \lambda}(0)=\beta>0$ and the definition of $t_*$ as the smallest root of $h_{ \beta, \lambda}$.
  Since $h_{ \beta, \lambda}$ is strictly convex,
  \begin{equation}
    \label{eq:aux.pr.1}
    0=h_{ \beta, \lambda}(t_*)>h_{ \beta, \lambda}(t)+h_{ \beta, \lambda}'(t)(t_*-t),\qquad t\in [0,R), \; t\neq t_*.
  \end{equation}
  If $t\in [0,t_*)$ then $h_{ \beta, \lambda}(t)>0$ and $t_*-t>0$, which, combined with
  \eqref{eq:aux.pr.1} yields $h_{ \beta, \lambda}'(t)<0$ for all $t\in [0,t_*)$. Hence, using $\lambda\geq0$ and
$h_{ \beta, \lambda}'(t)=\lambda+h_{ \beta, 0}'(t)$ for all $t\in [0,t_*)$ the second inequality in \eqref{eq:n.f} follows.
  The third inequality in \eqref{eq:n.f} follows from the first and the
  second inequalities.

To prove the last inequality in \eqref{eq:n.f}, note that  the
  division of the inequality on \eqref{eq:aux.pr.1} by $-h_{ \beta, \lambda}'(t)$ (which
  is strictly positive), together with some simple algebraic manipulations, gives
$$t-h_{ \beta, \lambda}(t)/h_{ \beta, \lambda}'(t)< t_*,
    \quad  \forall ~t\in [0,t_*),$$
which, using the first inequality in \eqref{eq:n.f} and  $0<-h_{ \beta, \lambda}'(t)\leq- h_{ \beta, 0}'(t)$  for all $t\in [0,t_*)$,
yields the desired inequality.

Since  $h_{ \beta, \lambda}>0$ in $[0, t_*)$ and $h_{ \beta, \lambda}(t_*)=0$, we must have $ h_{ \beta, \lambda}'(t_*)\leq 0$.
Thus, the last inequality of the proposition follows from the fact that  $h_{ \beta, \lambda}'(t_*)=\lambda+h_{ \beta,0}'(t_*)$.
\end{proof}
In view of the second inequality in (\ref{eq:n.f}), the following  iteration map for $h_{ \beta, \lambda}$  is
well defined in $[0,t_*)$. Denoting this by $n_{h_{ \beta, \lambda}}$:
\begin{equation} \label{eq:n.f.2}
  \begin{array}{rcl}
  n_{h_{ \beta, \lambda}}:[0,t_*)&\to& \mathbb{R}\\
    t&\mapsto& t-h_{\beta,\lambda}(t)/h_{\beta, 0}'(t).
  \end{array}
\end{equation}
Note that in the case where $\lambda=0$, the sequence $n_{h_{ \beta, \lambda}}$ reduces to a Newton sequence,
 which  Ferreira and Svaiter used in \cite{FS06} to obtain a semi-local convergence analysis of the Newton method under a majorant condition.
\begin{proposition} \label{pr:1.5}
  For  each $ t\in [0,t^*)$ it holds that $\beta \leq n_{h_{ \beta, \lambda}}(t)<t_{*}$.
  \end{proposition}
\begin{proof}
Proposition~\ref{pr:1} implies that $h_{ \beta, \lambda}$ is convex. Hence, using  item {\bf i} of Proposition~\ref{pr:10}
 it is easy to see, by using convexity properties, that $ (1-\lambda)t-\beta \geq - h_{ \beta, \lambda}(t)$, which combined with $\lambda \geq 0$ gives $t-\beta \geq - h_{ \beta, \lambda}(t)$. Accordingly, the above definition implies that
$$
 n_{h_{ \beta, \lambda}}(t)-\beta=t- \frac{h_{ \beta, \lambda}(t)}{h'_{ \beta, 0}(t)}-\beta \geq -h_{ \beta, \lambda}(t)- \frac{h_{ \beta, \lambda}(t)}{h'_{ \beta,0}(t)}=\frac{h_{ \beta, \lambda}(t)}{-h'_{ \beta, 0}(t)}[h'_{ \beta, 0}(t)+1], \qquad  \forall ~t\in [0,t_*) .
$$
Proposition~\ref{pr:10} implies that $h'_{ \beta, 0}(0)=-1$ and   $h'_{ \beta, 0}$ is strictly increasing. Thus, we obtain  $h'_{ \beta, 0}(t)+1\geq 0$, for all $t\in [0,t_*)$.
Therefore, combining the above inequality with the first two inequalities in Proposition~\ref{pr:1}, the first inequality of proposition follows.
To prove the last inequality of proposition, combine \eqref{eq:n.f.2} with the last inequality in \eqref{eq:n.f}.
\end{proof}
\begin{proposition} \label{pr:2}
 Iteration map  $n_{h_{ \beta, \lambda}}$  maps $[0,t^*)$ in
  $[0,t^*)$, and it holds that
$$  t<n_{h_{ \beta, \lambda}}(t), \qquad   \forall\,t \in [0,t_*).$$
Moreover, if $\lambda=0$ or $\lambda=0$ and  ${h_{ \beta, 0}'}(t_*)<0$, we have the follows inequalities, respectively,
  $$    t_*-n_{h_{ \beta, \lambda}}(t)\leqslant  \frac{1}{2}(t_*-t) ,\qquad
t_*-n_{h_{ \beta, \lambda}}(t) \leq \frac{D^{-}{h_{ \beta, 0}'(t_*)}}{-2h_{ \beta, 0}'(t_*)}  (t_*-t)^2, \qquad
    \forall\,t \in [0,t_*).$$
 \end{proposition}
\begin{proof}
The first two statements of the proposition follow trivially for the last inequalities in  \eqref{eq:n.f} and \eqref{eq:n.f.2}.
Now, if $\lambda=0$, then the sequence in \eqref{eq:n.f.2} reduces to a Newton sequence. Hence, the second part of the proof  follows the same pattern as the proof of Proposition 4 of \cite{FS06}
with $h_{ \beta, 0}=f$.
\end{proof}
The definition of $\{t_{k}\}$ in Theorem~\ref{th:knt} is equivalent to the following one
\begin{equation}
  \label{eq:tknk}
  t_{0}=0,\quad t_{k+1}=n_{h_{ \beta, \lambda}}(t_{k}), \qquad k=0,1,\ldots\, .
\end{equation}
Therefore, using also Proposition~\ref{pr:2} it is easy to prove  that
\begin{corollary} \label{cr:kanttk}
  The sequence $\{t_{k}\}$ is well defined, is strictly increasing,
  is contained in $[0,t_*)$, and converges to $t_*$.

Moreover, if $\lambda=0$ or $\lambda=0$ and  ${h_{ \beta, 0}'}(t_*)<0$, the sequence $\{t_{k}\}$ converges Q-linearly or Q-quadratically
to $t_*$, respectively,  as follows
$$
 t_*-t_{k+1}\leq \frac{1}{2}( t_*-t_{k}), \qquad  t_*-t_{k+1} \leq \frac{D^{-}{h_{ \beta, 0}'(t_*)}}{-2 h_{ \beta, 0}'(t_*)}(t_*-t_{k})^2,
    \quad   k=0,1,\,\ldots\,.
$$
\end{corollary}

Hence, all statements involving only $\{t_{k}\}$ on Theorem~\ref{th:knt} are valid

%%%%%%%%%%%%%%%%%%%%%%%%%%%%%%%%%%%%%%%%%%%%%%%%%%%%%%%%%%%%%%%5
\subsection{Convergence} \label{sec:MFNLO}
%%%%%%%%%%%%%%%%%%%%%%%%%%%%%%%%%%%%%%%%%%%%%%%%%%%%%%%%%%%%%%%%%%%%%%%%%%%%%%%%%%%%
In this section we will prove well definedness and convergence of the sequence  $\{x_{k}\}$ specified
 on \eqref{ns.KT} in Theorem~\ref{th:knt}.

We start with two lemma that highlight the  relationships between the majorant function $f$ and the non-linear function $F$.
%%%%%%%%%%%%%%%%%%%%%%%%%%%%%55
\begin{proposition} \label{wdns}
If \,$\| x-x_0\|\leq t <t_*$, then
$rank(F'(x))=rank(F'(x_0))\geq 1$  and
$$
\left\|F'(x)^{\dagger}\right\|\leq {-\| F'(x_0)^{\dagger}\|}/{h_{\beta, 0}'(t)}.  $$
In particular, $rank(F'(x))=rank(F'(x_0))$ in $B(x_0, t_*)$.
\end{proposition}
\begin{proof}
 Take $x\in B[x_0,t]$, $0\leq t<t_*$.  Using the assumptions
  \eqref{KH}, {\bf h1}, {\bf h2},  $f'(t)=h_{\beta,0}'(t)$ and the
  second inequality in \eqref{eq:n.f} we   obtain
$$
 \|F'(x_0)^{\dagger}\|\|F'(x)-F'(x_0)\|
    \leqslant  f'(\|x-x_0\|)-f'(0)\leqslant
     f'(t)+1=h_{ \beta,0}'(t)+1<1.
$$
  Combining the last inequality with  \eqref{posto} and  Lemma~\ref{lem:ban}, we conclude that
  $rank(F'(x))=rank(F'(x_0))\geq 1$  and
  \begin{align*}
    \|F'(x)^{\dagger}\| &\leqslant
    \frac{\| F'(x_0)^{\dagger}\|}{1-\left(f'(t)+1\right)}= \frac{\| F'(x_0)^{\dagger}\|}{-f'(t)}=-\frac{\| F'(x_0)^{\dagger}\|}{h_{\beta,0}'(t)}.
  \end{align*}
\end{proof}
It is convenient to study the linearization error of $F$ at point in~$\Omega$. For that purpose we define
\begin{equation}\label{eq:def.er}
  E_F(x,y):= F(y)-\left[ F(x)+F'(x)(y-x)\right],\qquad y,\, x\in \Omega.
\end{equation}
We will bound this error by the error in the linearization on the
majorant function $f$
\begin{equation}\label{eq:def.erf}
        e_f(t,u):= f(u)-\left[ f(t)+f'(t)(u-t)\right],\qquad t,\,u \in [0,R).
\end{equation}
%%%%%%%%%%%%

\begin{lemma}  \label{pr:taylor}
%%%%%%%%%%%%%
Take
$$
   x,y\in B(x_0,R) \quad\mbox{and}\quad 0\leq t<v<R.
$$
If $\|x-x_0\|\leqslant t$ and $\|y-x\|\leqslant v-t$, then
\[
\|F'(x_0)^{\dagger}\|\|E_F(x,y)\|\leqslant e_f(t, v)\frac{\|y-x\|^2}{(v-t)^2}.
\]
\end{lemma}
\begin{proof}
The proof follows the same pattern as the proof of Lemma~7  of \cite{FS06}.
\end{proof}

Proposition \ref{wdns} guarantees, in particular,  that  $rank(F'(x))\geq 1$ for all $x \in B(x_0,t_*)$
 and, consequently, the Gauss-Newton
iteration map is well-defined.  Let us call $G_{F}$, the
Gauss-Newton iteration map for $F$ in that region:
\begin{equation} \label{NF}
  \begin{array}{rcl}
  G_{F}:B(x_0, t_*) &\to& \mathbb{R}^n\\
    x&\mapsto& x- F'(x)^{\dagger}F(x).
  \end{array}
\end{equation}
One can apply a \emph{single} Gauss-Newton iteration on any $x\in
B(x_0,t_*)$ to obtain $G_{F}(x)$ which may not belong to $B(x_0,
t_*)$, or even may not belong to the domain of $F$. Therefore, this is
enough to guarantee well definedness of only one iteration. To
ensure that Gauss-Newton iterations may be repeated indefinitely,
we need the following result.

First, we define some subsets of $B(x_0, t_*)$ in which, as we shall
prove, the desired inclusion holds   for all points in these subsets.  %and consequently the Gauss-Newton iteration multifunction  is ``well behaved''.
\begin{align}\label{E:K}
K(t)&:=\left\{ x\in\Omega\, : \;  \|x-x_0\|\leq t, ~\|F'(x)^{\dagger}F(x)\| \leqslant -\frac{h_{\beta,\lambda}(t)}{h_{\beta,0}'(t)}\right\},\qquad
   t\in [0,t_*)\,,\\
  \label{eq:def.K}
 K&:=\bigcup_{t\in[0,t_*)} K(t).
\end{align}

In \eqref{E:K}, $0\leqslant t<t_*$, therefore, $h_{\beta,0}'(t)\neq 0$ and
$rank(F'(x))\geq 1$  in $B[x_0,t]\subset B[x_0,t_*)$  (Proposition \ref{wdns}).
Hence, the definitions are consistent.

\begin{lemma} \label{NfNF}
For each $t\in [0, t_*)$, it holds that:
\begin{itemize}
 \item[{\bf i)}]$K(t)\subset B(x_0,t_*)$;
  \item[{\bf  ii)}] $
 \|G_F(G_F(x))-G_F(x)\| \leq -\frac{h_{\beta,\lambda}(n_{h_{\beta,\lambda}}(t))}{{h'_{\beta,0}}(n_{h_{\beta,\lambda}}(t))}\left( \frac{\|G_F(x)-x\|}{n_{h_{\beta,\lambda}}(t)-t}\right)^2, \quad \forall \, x \in K(t),
 $
  \item[{\bf  iii)}]  $ G_F\left( K(t) \right)\subset K\left(n_{h_{ \beta, \lambda}}(t) \right).$
   \end{itemize}
As a consequence, $K\subset B(x_0,t_*)$ and $G_F(K)\subset K.$\\
\end{lemma}
\begin{proof}
Item {\bf i} follows trivially from the definition of $K(t)$.

Take $t\in[0,t_*)$, $x\in K(t)$.  Using definition \eqref{E:K} and the
first two statements in Proposition~\ref{pr:2} we have
\begin{equation}  \label{eq:eq.aux.k}
   \| x-x_0\|\leq t,\qquad \|F'(x)^{\dagger}F(x)\| \leq
-h_{\beta,\lambda}(t)/h_{\beta,0}'(t),\quad t<n_{h_{ \beta, \lambda}}(t)<t_*.
\end{equation}
Therefore
\begin{align*}%\label{eNfNF-1}
\|G_F(x)-x_0\|&\leqslant \|x-x_0\|+\| G_F(x)-x\|= \|x-x_0\|+\|F'(x)^{\dagger}F(x)\|\\
              &\leqslant t-h_{\beta,\lambda}(t)/h_{\beta,0}'(t)=n_{h_{ \beta, \lambda}}(t)<t_*\,,
\end{align*}
and
\begin{equation} \label{eq:auxNfNF}
 G_F(x)\in B[x_0,n_{h_{ \beta, \lambda}}(t)]\subset B(x_0, t_*).
\end{equation}
Since $G_F(x)$, $n_{h_{ \beta, \lambda}}(t)$ belong to the domains of $F$ and $f$, respectively,
using the definitions in \eqref{eq:n.f.2} and \eqref{NF},  $h_{ \beta,\lambda}(t)= \beta+\lambda t + f(t)$, linearization errors \eqref{eq:def.er} and \eqref{eq:def.erf} and some algebraic manipulation, we obtain
\begin{align}\label{rqs}
h_{ \beta, \lambda}(n_{h_{ \beta, \lambda}}(t))
&=h_{ \beta, \lambda}(n_{h_{ \beta, \lambda}}(t))-\left[h_{ \beta, \lambda}(t)+h_{ \beta,0}'(t)(n_{h_{ \beta, \lambda}}(t)-t)\right]\nonumber\\
          &=e_f(t,n_{h_{ \beta, \lambda}}(t))-\lambda h_{ \beta, \lambda}(t)/h_{ \beta,0}'(t)
\end{align}
and
\begin{align*}
        F(G_F(x))
        &=F(G_F(x))-\left[F(x)+F'(x)(G_F(x)-x)\right]+(I_{\mathbb{R}^m}-F'(x)F'(x)^{\dagger})F(x)\\
        &=E_F(x,G_F(x))+(I_{\mathbb{R}^m}-F'(x)F'(x)^{\dagger})F(x).
\end{align*}
The last equation, together with simple algebraic manipulations, implies that
\begin{multline*}
\|F'(G_F(x))^{\dagger}F(G_F(x))\|\leq \|F'(G_F(x))^{\dagger}\|\|E_F(x,G_F(x))\| \\+\|F'(G_F(x))^{\dagger}(I_{\mathbb{R}^m}-F'(x)F'(x)^{\dagger})F(x)\|.
\end{multline*}
As $\|G_F(x)-x_0\|\leq n_{h_{ \beta, \lambda}}(t)$, it follows from Proposition \ref{wdns} that $rank (F'(G_F(x)))\geq 1$  and
\[ \|F'(G_F(x))^{\dagger}\|\leq -\|F'(x_0)^{\dagger}\|/h_{ \beta, 0}'(n_{h_{ \beta, \lambda}}(t)).\]
From the two latter equations and \eqref{rank} we have
$$\|F'(G_F(x))^{\dagger}F(G_F(x))\|\leq -\frac{\|F'(x_0)^{\dagger}\|}{h_{ \beta, 0}'(n_{h_{ \beta,\lambda}}(t))}\|E(x,G_F(x))\| +\kappa\|G_F(x)-x\|.$$
On the other hand, using \eqref{eq:eq.aux.k}, Lemma \ref{pr:taylor} and \eqref{rqs} we have
\begin{align*}
\| F'(x_0)^{\dagger}\|\|E_F(x,G_F(x)) \|& \leq e_f(t,n_{h_{ \beta, \lambda}}(t))\left( \frac{\|G_F(x)-x\|}{n_{h_{ \beta,\lambda}}(t)-t}\right)^2\\
&\leq h_{ \beta, \lambda}(n_{h_{ \beta, \lambda}}(t))\left( \frac{\|G_F(x)-x\|}{n_{h_{ \beta,\lambda}}(t)-t}\right)^2 +\lambda h_{ \beta, \lambda}(t)/h_{ \beta,0}'(t).
\end{align*}
Thus, the last two equations, together with the second equation in \eqref{eq:eq.aux.k}, imply
\begin{align*}
\|F'(G_F(x))^{\dagger}F(G_F(x))\|&\leq \frac{-h_{ \beta, \lambda}(n_{h_{ \beta, \lambda}}(t))}{h_{ \beta, 0}'(n_{h_{ \beta,\lambda}}(t))}\left( \frac{\|G_F(x)-x\|}{n_{h_{ \beta,\lambda}}(t)-t}\right)^2 \\
&+(\kappa+\lambda{h_{ \beta, 0}'(n_{h_{ \beta,\lambda}}(t))}^{-1})(-h_{ \beta, \lambda}(t)/h_{ \beta,0}'(t)).
\end{align*}
Taking $\lambda\geq -\kappa f'(\beta)$, the second inequality in \eqref{eq:n.f} and \eqref{eq:eq.aux.k}, we obtain
 $$\big(\kappa+\lambda({h_{ \beta, 0}'(n_{h_{ \beta,\lambda}}(t))})^{-1}\big)\leq \kappa\big(1-f'(\beta)({h_{ \beta, 0}'(n_{h_{ \beta,\lambda}}(t))}^{-1}\big).$$
As $f'(t)={h'_{\beta,0}}(t)$, using Proposition~\ref{pr:1.5}, {\bf h2} and the second inequality in \eqref{eq:n.f}, we have
$$
\kappa\big(1-f'(\beta)({h_{ \beta, 0}'(n_{h_{ \beta,\lambda}}(t))}^{-1}\big)=\kappa\big(h_{ \beta, 0}'(\beta)-h_{ \beta, 0}'(n_{h_{ \beta,\lambda}}(t))\big)(-{h_{ \beta, 0}'(n_{h_{ \beta,\lambda}}(t))}^{-1}\leq 0.
$$
Combining the three above inequalities we conclude
$$
\|F'(G_F(x))^{\dagger}F(G_F(x))\| \leq \frac{-h_{ \beta, \lambda}(n_{h_{ \beta, \lambda}}(t))}{h_{ \beta, 0}'(n_{h_{ \beta,\lambda}}(t))}\left( \frac{\|G_F(x)-x\|}{n_{h_{ \beta,\lambda}}(t)-t}\right)^2.
$$
Therefore, item {\bf ii} follows from the last inequality and \eqref{NF}. Now, the last inequality combined with \eqref{eq:n.f.2}, \eqref{NF} and the second inequality in \eqref{eq:eq.aux.k} becomes
$$
\|F'(G_F(x))^{\dagger}F(G_F(x))\| \leq \frac{-h_{ \beta, \lambda}(n_{h_{ \beta, \lambda}}(t))}{h_{ \beta, 0}'(n_{h_{ \beta,\lambda}}(t))}.
$$
This result, together with \eqref{eq:auxNfNF}, shows that $G_F(x)\in
K(n_{h_{ \beta, \lambda}}(t))$, which proves item {\bf iii}.

The next inclusion (first on the second part), follows trivially
from definitions \eqref{E:K} and \eqref{eq:def.K}. To check the last
inclusion, take $x\in K$.  Then $x\in K(t)$ for some $t\in [0,t_*)$.
Using item {\bf iii} of the lemma, we conclude that $G_F(x)\in
K(n_{h_{ \beta,\lambda}}(t))$. To end the proof, note that $n_{h_{ \beta, \lambda}}(t)\in [0,t_*)$ and use
the definition of $K$.
\end{proof}
Finally, we are ready to prove the main result of this section, which
is an immediate consequence of the latter result. First note that the
sequence $\{x_k\}$ (see \eqref{ns.KT}) satisfies
\begin{equation} \label{NFS}
x_{k+1}=G_F(x_k),\qquad k=0,1,\ldots \,,
\end{equation}
which is indeed an equivalent definition of this sequence.
%%%%%%%%%%%%%%%%%%%%%%%%%%%5
\begin{corollary}\label{auxcor}
The sequence $\{x_k\}$ is well defined, is contained in $B(x_0,t_*)$,
converges to a point $x_*\in B[x_0,t_*]$ such that $F'(x_*)^{\dagger}F(x_*)=0$, and $\{x_k\}$ and
$\{t_k\}$ satisfy \eqref{eq:bd} and \eqref{eq:002}.

 Moreover, if $\lambda=0$ ( $\lambda=0$ and  ${h_{ \beta, 0}'}(t_*)<0$), the sequences $\{t_k\}$ and $\{x_k\}$ converge Q-linearly and R-linearly
  (Q-quadratically and R-quadratically) to $t_{*}$ and  $x_{*}$, respectively.
\end{corollary}
\begin{proof}
Since $\|F'(x_0)^{\dagger}F(x_0)\|=\beta$, using the  item {\bf i} of the Proposition~\ref{pr:10}, we have
\[  x_0\in K(0)\subset K,\]
where the second inclusion follows trivially from \eqref{eq:def.K}.  Using the above equation,
the inclusions $G_F(K)\subset K$ (Lemma \ref{NfNF}) and \eqref{NFS}, we conclude that the sequence
$\{x_k\}$ is well defined and lies in $K$. From the first inclusion in the second part of Lemma \ref{NfNF},
we have trivially that $\{x_k\}$ is contained in $B(x_0, t_*)$.

We will prove, by induction, that
\begin{equation}
        \label{eq:xktk}
        x_k\in K(t_k), \qquad k=0,1,\ldots \,.
\end{equation}
The above inclusion, for  $k=0$, is the first result in this proof.
Assume now that $x_k\in K(t_k)$. Thus, using item {\bf iii} of Lemma \ref{NfNF}, \eqref{eq:tknk} and \eqref{NFS}, we conclude that $x_{k+1}\in K(t_{k+1}),$, which completes the induction proof of \eqref{eq:xktk}.

Now, using
\eqref{eq:xktk} and \eqref{E:K}, we have
\[
  \|F'(x_k) ^{\dagger}F(x_k)\|\leq -{h_{\beta, \lambda}}(t_k)/{h'_{\beta,0}}(t_k),\qquad k=0,1,\ldots\,,
\]
which, using \eqref{ns.KT}, becomes
\begin{equation}
        \label{eq:xktx-2}
         \|x_{k+1}-x_k\|\leq t_{k+1}-t_k, \qquad k=0,1,\ldots\,.
\end{equation}
So, the first inequality in \eqref{eq:bd} holds. As $\{t_k\}$ converges to $t_*$, the last inequality implies
that
\[
   \sum_{k=k_0}^\infty \|x_{k+1}-x_k\|\leq
   \sum_{k=k_0}^\infty t_{k+1}-t_k =t_*-t_{k_0}<+\infty,
\]
for any $k_0\in\mathbb{N}$. Hence, $\{x_k\}$ is a Cauchy sequence in
$B(x_0, t_*)$, and so converges to some $x_*\in B[x_0,t_*]$.
The last inequality also implies that the second  inequality in \eqref{eq:bd} holds.

To prove that $F'(x_*)^{\dagger}F(x_*)=0$, note that, with simple algebraic manipulation, \eqref{rank} and
\eqref{ns.KT}, we obtain
\begin{align*}
\|F'(x_*)^{\dagger}F(x_k)\|&\leqslant \|F'(x_*)^{\dagger}\big(I-F'(x_k)F'(x_k)^{\dagger}\big)F(x_k)\|\\
&+ \|F'(x_*)^{\dagger}\|\|F'(x_k)F'(x_k)^{\dagger}F(x_k)\|\\
&\leq \kappa\|x_k-x_*\|+ \|F'(x_*)^{\dagger}\|\|F'(x_k)\|\|x_{k+1}-x_k\|.
\end{align*}
Due the fact that $F$ is continuously differentiable, we can take limit in the last inequality to conclude that
$F'(x_*)^{\dagger}F(x_*)=0$.

Since $  x_{k}\in K(t_{k})$, for all $k=0,1,\ldots,$ the inequality in \eqref{eq:002}, follows by
 applying item~{\bf ii} of the Lemma~\ref{NfNF} with $x=x_{k-1}$ and $t=t_{k-1}$ and using the definitions in \eqref{eq:tknk} and \eqref{NFS}.

To end the proof, combined the second inequality in \eqref{eq:bd}   with the last part of the Corollary~\ref{cr:kanttk}.
\end{proof}
Therefore, it follows from Corollaries \ref{cr:kanttk} and \ref{auxcor} that all statements in Theorem~\ref{th:knt} are valid.
%%%%%%%%%%%%%%%%%%%%%%%%%5
%%%%%%%%%%%

%%%%%%%%%%%%%%%%%%%%%%
\section{Special cases} \label{sec:ec}
%%%%%%%%%%%%%%%%%%%%%%%%%%%%%%%%%%%%%%%
In this section, we present some special cases of Theorem~\ref{th:knt}.
%%%%%%%%%%%%%%%%%%%%%%%%%%%%%%%%%%%%%%%%
%%%%%%%%%%%%%%%%%%%%
\subsection{Convergence result for $F'(x_0)$ surjective}
%%%%%%%%%%%%%%%%%%%%%%%%%55
In this section we present a  theorem  under the  hypothesis that  $F'(x_0)$ is surjective.
In this case, we can use a  majorant condition, which gives the propriety that $\{x_k\}$ is invariant under  the function $\bar{F}\rightarrow A^{\dagger}F$, where
$A: \mathbb{R}^n \to \mathbb{R}^m$ is any surjective linear operator.
%%%%%%%%%%%%%%%%%%%%%
\begin{theorem}\label{th:kntmod}
Let $\Omega\subseteq\mathbb{R}^n$ be an open set and $F:\Omega\to \mathbb{R}^m$ a continuously differentiable function. Take
 $x_0\in\Omega$ such that $\beta:=\|F'(x_0)^{\dagger}F(x_0)\|>0$ and  $F'(x_0)$ is surjective. Suppose that there exist $R>0$ and  a continuously differentiable function  $f:[0,\;  R)\to \mathbb{R}$  such that,
$B(x_0,R)\subseteq \Omega$,
  \begin{equation}\label{KH1}
 \|F'(x_0)^{\dagger}(F'(y)-F'(x))\| \leq \bar{f}'(\|y-x\|+\|x-x_0\|)-\bar{f}'(\|x-x_0\|),
  \end{equation}
   for any $x,y\in B(x_0, R)$,  $\|x-x_0\|+\|y-x\|< R$,
  \begin{itemize}
  \item[{\bf h1)}] $\bar{f}(0)=0$, $\bar{f'}(0)=-1$;
  \item[{\bf h2)}] $\bar{f'}$ is convex and strictly increasing.
  \end{itemize}
 Consider the auxiliary function  $h_{ \beta}:[0,\;  R)\to \mathbb{R}$,
 \begin{equation}\label{defh1}
h_{ \beta}(t):=\beta + \bar{f}(t).
\end{equation}
If  $h_{ \beta}$   satisfies
\begin{itemize}
  \item[{\bf h3)}]   $h_{ \beta}(t)=0$   for some  $t \in (0,  R)$,
\end{itemize}
then $h_{ \beta}(t)$ has a smallest zero   $\bar{t}_* \in (0,  R)$,  the  sequences
for solving  $ h_{ \beta}(t)=0$ and $F(x)=~0$, with starting point $t_0=0$ and $x_0$,  respectively,
\begin{equation}\label{ns.KT1}
    t_{k+1} ={t_{k}}-h_{ \beta}'(t_{k}) ^{-1}h_{ \beta}(t_{k}),\quad x_{k+1}={x_k}-F'(x_k)^{\dagger}F(x_k), \quad k=0,1,\ldots\,,
\end{equation}
are well defined,  $\{t_{k}\}$ is strictly increasing, is contained in
  $[0,t_*)$, and converges Q-linearly to $\bar{t}_*$, $\{x_k\}$ is contained in $B(x_0,\bar{t}_*)$, and converges R-linearly to a point $x_*\in B[x_0, \bar{t}_*]$ such that
$F'(x_*)^{\dagger}F(x_*)=0$,
\begin{equation}\label{2020}
  \|x_{k+1}-x_{k}\|   \leq  t_{k+1}-t_{k} , \qquad \|x_*-x_{k}\|   \leq  \bar{t}_*-t_{k},\qquad  k=0, 1, \ldots\, ,
\end{equation}
$$
\|x_{k+1}-x_{k}\|\leq   \frac{t_{k+1}-t_{k}}{(t_{k}-t_{k-1})^2} \|x_k-x_{k-1}\|^2,
 \quad k=1,2, \ldots\,,
$$
\begin{equation}\label{eq:bd1112}
  \|F'(x_0)^{\dagger}F(x_k)\|\leq \left(  \frac{t_{k+1}-t_{k}}{t_{k}-t_{k-1}}\right)\|F'(x_0)^{\dagger}F(x_{k-1})\| ,\quad k=1,2, \ldots\,.
  \end{equation}
If, additionally, ${h_{ \beta}'}(t_*)<0$, then the sequences $\{t_{k}\}$ and $\{x_{k}\}$ converge Q-quadratically and R-quadratically to $t_{*}$ and  $x_{*}$, respectively.
\end{theorem}
\begin{proof}
Let  $\bar{F}:\Omega \to \mathbb{R}^m$ be defined by
\begin{equation}\label{1010}
 \bar{F}(x)=F'(x_0)^{\dagger}F(x), \qquad x \in \Omega.
\end{equation}
Under the hypothesis of the theorem, we will prove that $\bar{F}$ satisfies all assumptions of the Theorem \ref{th:knt}.  Hence,
with the exception of \eqref{eq:bd1112}, the statements  of the theorem follow from Theorem~\ref{th:knt}.

First of all, as $F'(x_0)$ is surjective, it follows from \eqref{121212} that
\begin{equation}\label{1212}
F'(x_0)F'(x_0)^{\dagger}=I_{\mathbb{R}^m}, \qquad (F'(x_0)F'(x_0)^{\dagger})^{\dagger}=F'(x_0)F'(x_0)^{\dagger}.
\end{equation}
Now, take $x\in B[x_0,t]$, $0\leq t\leq \bar{t}_*$.  Using the assumptions
  \eqref{KH1}, {\bf h1} and  {\bf h2},   we   obtain
$$
 \|F'(x_0)^{\dagger}[F'(x)-F'(x_0)]\|
    \leqslant  \bar{f}'(\|x-x_0\|)-\bar{f}'(0)\leqslant
     \bar{f}'(t)+1<1.
$$
Using Lemma~\ref{lem:ban1}, the above equation and the first equation in \eqref{1212}, we conclude that
  $\big(I_{\mathbb{R}^n}-F'(x_0)^{\dagger}(F'(x_0)-F'(x))\big)$ is non-singular   and
  \begin{equation}\label{1012}     \|\big(I_{\mathbb{R}^n}-F'(x_0)^{\dagger}(F'(x_0)-F'(x))\big)^{-1}\| \leqslant
    \frac{1}{1-\left(\bar{f}'(t)+1\right)}=-\frac{1}{\bar{f}'(t)}.
\end{equation}
Again, the first equation in \eqref{1212} implies that
$F'(x)=F'(x_0)(I_{\mathbb{R}^n}-F'(x_0)^{\dagger}(F'(x_0)-F'(x))),$ which, using $F'(x_0)$ is surjective
and  $\big(I_{\mathbb{R}^n}-F'(x_0)^{\dagger}(F'(x_0)-F'(x))\big)$ is non-singular, yields $F'(x)$ is surjective for all
$x \in B(x_0,\bar{t}_*).$ % So, $\bar{F'}$ satisfy  \eqref{posto}.
Hence, using \eqref{1010} and properties of the Moore-Penrose inverse, we have
$$(\bar{F'}(x))^{\dagger}=(F'(x_0)^{\dagger}F'(x))^{\dagger}=F'(x)^{\dagger}F'(x_0),\qquad \forall \; x \in \Omega. $$
The latter inequality implies that $\bar{F'}$ satisfies \eqref{rank} with $\kappa=0$ and the second sequence in \eqref{ns.KT1} coincides with
the second sequence in \eqref{ns.KT}. Moreover, using \eqref{1010},   \eqref{1212} and \eqref{1111}, we obtain
\begin{equation}\label{1515}
\|\bar{F'}(x_0)^{\dagger} \bar{F}'(x_0)\|=\|({F'}(x_0)^{\dagger} {F}'(x_0))^{\dagger}{F'}(x_0)^{\dagger} {F}'(x_0)\|=\|{F'}(x_0)^{\dagger} {F}'(x_0)\|
\end{equation}
and
\begin{equation}\label{15151}
\|\bar{F'}(x_0)^{\dagger}\|=\|{F'}(x_0)^{\dagger} {F}'(x_0)\|=\| \Pi_{Ker({F}'(x_0))^{\bot}}\|=1.
\end{equation}
Accordingly, \eqref{1515} implies that $\|\bar{F'}(x_0)^{\dagger} \bar{F}'(x_0)\|>0$, and \eqref{15151} together with \eqref{KH1} and \eqref{1010} implies that $\bar{F'}$ satisfies~\eqref{KH} with $f=\bar{f}$.

Therefore, with the exception \eqref{eq:bd1112}, the result of the theorem follow from Theorem~\ref{th:knt} with $F=\bar{F},$ $f=\bar{f}$, $h_{ \beta, \lambda}=h_{ \beta}$, $\lambda=0 $ and $t_*=\bar{t}_* $.

Our task is now to show that  \eqref{eq:bd1112} holds.

Take $k \in \{1,2,\ldots\,\}$. Using the first equation in \eqref{1212}, it follows by simple calculus that
$$F'(x_{k-1})^{\dagger}F'(x_0)\big(I_{\mathbb{R}^n}-F'(x_0)^{\dagger}(F'(x_0)-F'(x_{k-1}))\big)=F'(x_{k-1})^{\dagger}F'(x_{k-1}),$$
which, combined with \eqref{1111}, \eqref{1012}  and $\|x_{k-1}-x_{0}\|   \leq  t_{k-1}  \leq \bar{t}_*$,  yields
\begin{align*}
\|F'(x_{k-1})^{\dagger}F'(x_0)\|&\leq  \|\Pi_{Ker({F}'(x_{k-1}))^{\bot}}(I_{\mathbb{R}^n}-F'(x_0)^{\dagger}(F'(x_0)-F'(x_{k-1}))\big)^{-1}\|\\
&\leq  \|(I_{\mathbb{R}^n}-F'(x_0)^{\dagger}(F'(x_0)-F'(x_{k-1}))\big)^{-1}\|\\
&\leq -(h'_{\beta}(t_{k-1}))^{-1}.
\end{align*}
Hence, using \eqref{ns.KT1} and the first equation in \eqref{1212}, we obtain
\begin{equation}\label{5050}
 \|x_{k}-x_{k-1}\|=\|F'(x_{k-1})^{\dagger}F(x_{k-1})\|\leq-(h'_{\beta}(t_{k-1}))^{-1}\|F'(x_0)^{\dagger}F(x_{k-1})\|.
 \end{equation}
Since $F(x_{k-1})$ is also surjective, it follows from \eqref{121212} that $F'(x_{k-1})F'(x_{k-1})^{\dagger}=I_{\mathbb{R}^m}$, which combined with
Lemma~\ref{pr:taylor} and \eqref{2020} gives
\begin{align*}
\|F'(x_{0})^{\dagger}F(x_k)\|&= \|F'(x_0)^{\dagger}(F(x_k)-F(x_{k-1})-F'(x_{k-1})(x_k-x_{k-1})\|\\
                             &=\|F'(x_0)^{\dagger}\|\|E_F(x_{k-1},x_k)\|\\
                             &\leq e_f(t_{k-1},t_k)\frac{\|x_k-x_{k-1}\|}{(t_{k}-t_{k-1})}\\
                             &=h_{\beta}(t_{k})\frac{\|x_k-x_{k-1}\|}{(t_{k}-t_{k-1})},
\end{align*}
where the latter equation is obtained by combining \eqref{eq:def.erf}, \eqref{defh1} and \eqref{ns.KT1}. Taking into account
the last inequality, \eqref{5050}, $\{t_{k}\}$ and $h'_{\beta}$ are strictly increasing,  we have
\begin{align*}
\|F'(x_{0})^{\dagger}F(x_k)\|  &\leq -\frac{h_{\beta}(t_{k})}{h'_{\beta}(t_{k-1})}\frac{\|F'(x_{0})^{\dagger}F(x_{k-1})\|}{(t_{k}-t_{k-1})}\\
                               &\leq -\frac{h_{\beta}(t_{k})}{h'_{\beta}(t_{k})}\frac{\|F'(x_{0})^{\dagger}F(x_{k-1})\|}{(t_{k}-t_{k-1})}.
\end{align*}
Therefore, the last inequality, together with the definition of $\{t_{k}\}$ in \eqref{ns.KT1}, imply the desired inequality.
\end{proof}

%%%%%%%%%%%%%%%%%%
\subsection{Convergence result for Lipschitz condition}
In this section, we first present a theorem corresponding to Theorem
\ref{th:knt}, but under the Lipschitz condition instead of the general assumption
\eqref{KH}. We also present a theorem corresponding to Theorem~\eqref{th:kntmod}, but under the Lipschitz condition instead of assumption
\eqref{KH1}.
%%%%%%%%%%%%%%%%%%%%
\begin{theorem}\label{th:kntL}
Let $\Omega\subseteq\mathbb{R}^n$ be an open set and $F:\Omega\to \mathbb{R}^m$ a continuously differentiable  function. Suppose that
 $$
\left\|F'(y)^{\dagger}(I_{\mathbb{R}^m}-F'(x)F'(x)^{\dagger})F(x)\right\| \leq \kappa \|x-y\|,\qquad \forall \;x,y \;\in \Omega
$$
for some $0 \leq \kappa < 1$. Take  $x_0\in\Omega$ such that $\beta:=\|F'(x_0)^{\dagger}F(x_0)\|>0$,  $F'(x_0)\neq 0$ and
 $$
  rank(F'(x))\leq rank(F'(x_0)), \qquad \forall \; x \in \Omega.
$$
Suppose that there exist $R>0$ and  $L>0$, such that $B(x_0,R)\subseteq \Omega$,
$$
\|F'(x_0)^{\dagger}\|\|F'(x)-F'(y)\| \leq L\|x-y\|, \qquad \forall\; x, y \in B(x_0, R)
$$
Take $\lambda=(1-\beta L)\kappa$ and consider the auxiliary function  $h_{ \beta, \lambda}:[0,\;  R)\to \mathbb{R}$,
$$
h_{ \beta, \lambda}(t):= \beta -(1- \lambda)t+ (Lt^2)/2.
$$
If
$$ \beta L\leq \Delta:=\frac{(1-\kappa)^2}{(\kappa^2-\kappa+1)+\sqrt{2\kappa^2-2\kappa+1}},$$
then $h_{ \beta, \lambda}(t)$ has a smallest zero   $t_*=\big(1-\lambda-\sqrt{(1-\lambda)^2 -2\beta L}\big)/L$,  the  sequences
for solving  $ h_{ \beta, \lambda}(t)=0$ and $F(x)=~0$, with starting point $t_0=0$ and $x_0$,  respectively,
$$
    t_{k+1} ={t_{k}}-h_{ \beta,0}'(t_{k}) ^{-1}h_{ \beta, \lambda}(t_{k}),\quad x_{k+1}={x_k}-F'(x_k)^{\dagger}F(x_k), \quad k=0,1,\ldots\,,
$$
are well defined,  $\{t_{k}\}$ is strictly increasing, is contained in
  $[0,t_*)$, and converges  to $t_*$, $\{x_k\}$ is contained in $B(x_0,t_*)$, converges to a point  $x_*\in B[x_0, t_*]$ such that
$F'(x_*)^{\dagger}F(x_*)=0$  and
$$
  \|x_{k+1}-x_{k}\|   \leq  t_{k+1}-t_{k} , \qquad \|x_*-x_{k}\|   \leq  t_*-t_{k},\qquad  k=0, 1, \ldots\, ,
 $$
 $$
\|x_{k+1}-x_{k}\|\leq   \frac{t_{k+1}-t_{k}}{(t_{k}-t_{k-1})^2} \|x_k-x_{k-1}\|^2,
 \quad k=1,2, \ldots\,.
$$
Moreover, if $\lambda=0$ ($\lambda=0$ and  ${h_{ \beta, 0}'}(t_*)<0$), then the sequences $\{t_k\}$ and $\{x_k\}$ converge Q-linearly and R-linearly
 (Q-quadratically and R-quadratically) to $t_{*}$ and $x_{*}$, respectively.
\end{theorem}
\begin{proof}
It is immediate to prove that $F$, $x_0$ and $f:[0,R)\to \mathbb{R}$ defined by $ f(t)=Lt^{2}/2-t, $ satisfy the
inequality \eqref{KH}, conditions {\bf h1} and {\bf h2}.  Hence,
$$
h_{ \beta, \lambda}(t):= \beta -(1- \lambda)t+ (Lt^2)/2=\beta + \lambda  t+ f(t).
$$
Since,
\begin{equation}\label{s12}
\beta L\leq\Delta=\frac{(1-\kappa)^2}{(\kappa^2-\kappa+1)+\sqrt{2\kappa^2-2\kappa+1}}=\frac{(1-\kappa)^2}{(1-\kappa)^2+\kappa+\sqrt{2\kappa^2-2\kappa+1}}\leq1,
\end{equation}
we have   $\lambda=(1-\beta L)\kappa\geq0$ and $\lambda=-\kappa f'(\beta)$. Moreover, the first inequality in \eqref{s12} implies that $(1-\lambda)^2-2\beta L\geq0$, i.e.,  $h_{ \beta, \lambda}$ satisfies {\bf h3} and   $t_*=\big(1-\lambda-\sqrt{(1-\lambda)^2 -2\beta L}\big)/L$ is its smallest root.

Therefore, taking $f$, $h_{ \beta, \lambda}$, $\lambda$ and $t_*$  as defined above, all  the statements  of the theorem follow from Theorem~\ref{th:knt}.
\end{proof}

Under the Lipschitz condition, Theorem~\ref{th:kntmod} becomes:

\begin{theorem}\label{th:kntL12}
Let $\Omega\subseteq\mathbb{R}^n$ be an open set and $F:\Omega\to \mathbb{R}^m$ a continuously differentiable  function. Take
 $x_0\in\Omega$ such that $\beta:=\|F'(x_0)^{\dagger}F(x_0)\|>0$ and  $F'(x_0)$ is surjective.
Suppose that there exist $R>0$ and  $L>0$, such that $B(x_0,R)\subseteq \Omega$,
$$
\|F'(x_0)^{\dagger}(F'(x)-F'(y))\| \leq L\|x-y\|, \qquad \forall\; x, y \in B(x_0, R)
$$
Consider the auxiliary function  $h_{ \beta}:[0,\;  R)\to \mathbb{R}$,
$$
h_{ \beta}(t):= \beta -t +(Lt^2)/2.
$$
If
$ \beta L\leq {1}/{2},$
then $h_{ \beta}(t)$ has a smallest zero   $\bar{t_*}=\big(1-\sqrt{1 -2\beta L}\big)/L$,  the  sequences
for solving  $ h_{ \beta}(t)=0$ and $F(x)=~0$, with starting point $t_0=0$ and $x_0$,  respectively,
$$
    t_{k+1} ={t_{k}}-h_{ \beta}'(t_{k}) ^{-1}h_{ \beta}(t_{k}),\quad x_{k+1}={x_k}-F'(x_k)^{\dagger}F(x_k), \quad k=0,1,\ldots\,,
$$
are well defined,  $\{t_{k}\}$ is strictly increasing, is contained in
  $[0,t_*)$, and converges Q-linearly to $\bar{t}_*$, $\{x_k\}$ is contained in $B(x_0,\bar{t}_*)$, and converges R-linearly to a point  $x_*\in B[x_0, \bar{t}_*]$ such that
$F'(x_*)^{\dagger}F(x_*)=0$,
$$
  \|x_{k+1}-x_{k}\|   \leq  t_{k+1}-t_{k} , \qquad \|x_*-x_{k}\|   \leq  \bar{t}_*-t_{k},\qquad  k=0, 1, \ldots\, ,
$$
$$
\|x_{k+1}-x_{k}\|\leq   \frac{t_{k+1}-t_{k}}{(t_{k}-t_{k-1})^2} \|x_k-x_{k-1}\|^2,
 \quad k=1,2, \ldots\,,
$$
$$
  \|F'(x_0)^{\dagger}F(x_k)\|\leq \left(  \frac{t_{k+1}-t_{k}}{t_{k}-t_{k-1}}\right)\|F'(x_0)^{\dagger}F(x_{k-1})\| ,\quad k=1,2, \ldots\,.
$$
If, additionally, $ \beta L< {1}/{2},$, then the sequences $\{t_{k}\}$ and $\{x_{k}\}$ converge Q-quadratically and R-quadratically to $t_{*}$ and  $x_{*}$, respectively.
\end{theorem}
\begin{proof}
The proof follows the same pattern as the proof of the Theorem~\ref{th:kntL}.
\end{proof}
%%%%%%%%%%%%%%%%%%%%%%%%%%%%%%%%%%%%%%%%%%%%%%%%%%%%%%%%%%%%%
\subsection{Convergence result under Smale's condition }
%%%%%%%%%%%%%%%%%%%%%%%%%%%%%%%%%%%%%%%%%%%%%%%%%%%%%%%%%%%%%
In this section, we first present a theorem corresponding to Theorem
\ref{th:knt}, but under Smale's $\alpha$-condition, see \cite{Dedieu2002,MR16517501,S86}.
 We also present a theorem corresponding to Theorem~\eqref{th:kntmod}, but under Smale's $\alpha$-condition instead of the assumption
\eqref{KH1}.

To simplify, we take $\lambda=\kappa$ in the next theorem. As seen in Remark~\ref{lambda}, this is always a possible choice for $\lambda$.

\begin{theorem}\label{theo:Smale}
Let $\Omega\subseteq\mathbb{R}^n$ be an open set and $F:\Omega\to \mathbb{R}^m$ an analytic function. Suppose that
 $$
\left\|F'(y)^{\dagger}(I_{\mathbb{R}^m}-F'(x)F'(x)^{\dagger})F(x)\right\| \leq \kappa \|x-y\|,\qquad \forall \;x,y \;\in \Omega
$$
for some $0 \leq \kappa < 1$. Take  $x_0\in\Omega$ such that $\beta:=\|F'(x_0)^{\dagger}F(x_0)\|>0$,  $F'(x_0)\neq 0$ and
 $$
  rank(F'(x))\leq rank(F'(x_0)), \qquad \forall \; x \in \Omega.
$$
Suppose that
\begin{equation} \label{eq:SmaleCond}
  \gamma := \|F'(x_0)^{\dagger}\|\sup _{ n > 1 }\left\| \frac
{F^{(n)}(x_0)}{n !}\right\|^{1/(n-1)}<+\infty, \qquad B(x_0,1/\gamma)\subseteq \Omega.
\end{equation}
Consider the auxiliary function   $h_{ \beta, \kappa}:[0,\; 1/\gamma)\to \mathbb{R}$,
$$
h_{ \beta, \kappa}(t):= \beta -(2- \kappa)t+ t/(1-\gamma t).
$$
If $$ \alpha:=\beta \gamma\leq3-2\sqrt{2},$$
then $h_{ \beta, \kappa}(t)$ has a smallest zero $t_*=\big(1-\kappa+\alpha-\sqrt{(1-\kappa+\alpha)^2 -4(2-\kappa)\alpha}\big)/(2\gamma(2-\kappa))$,  the  sequences
for solving $ h_{ \beta, \kappa}(t)=0$ and $F(x)=~0$, with starting point $t_0=0$ and $x_0$, respectively,
$$
    t_{k+1} ={t_{k}}-h_{ \beta,0}'(t_{k}) ^{-1}h_{ \beta, \kappa}(t_{k}),\quad x_{k+1}={x_k}-F'(x_k)^{\dagger}F(x_k), \quad k=0,1,\ldots\,,
$$
are well defined,  $\{t_{k}\}$ is strictly increasing, is contained in
  $[0,t_*)$, and converges  to $t_*$, $\{x_k\}$ is contained in $B(x_0,t_*)$, converges to a point  $x_*\in B[x_0, t_*]$ such that
$F'(x_*)^{\dagger}F(x_*)=0$ and
$$
  \|x_{k+1}-x_{k}\|   \leq  t_{k+1}-t_{k} , \qquad \|x_*-x_{k}\|   \leq  t_*-t_{k},\qquad  k=0, 1, \ldots\, ,
$$
$$
\|x_{k+1}-x_{k}\|\leq   \frac{t_{k+1}-t_{k}}{(t_{k}-t_{k-1})^2} \|x_k-x_{k-1}\|^2,
 \quad k=1,2, \ldots\,.
$$
Moreover, if $\kappa=0$ ($\kappa=0$ and ${h_{ \beta, 0}'}(t_*)<0$), then the sequences $\{t_k\}$ and $\{x_k\}$ converge Q-linearly and R-linearly
 (Q-quadratically and R-quadratically) to $t_{*}$ and $x_{*}$, respectively.
\end{theorem}
We need the following results to prove the above theorem.
\begin{lemma} \label{lemma:qc1}
Let $\Omega\subseteq\mathbb{R}^n$ be an open set and $F:\Omega\to \mathbb{R}^m$ an analytic function.
 Suppose that
$x_0\in \mathbb{R}^n $  and  $\gamma$ is defined in
\eqref{eq:SmaleCond}. Then, for all $x\in B(x_{0}, 1/\gamma)$
it holds that
$$
\|F'(x_0)^{\dagger}\|\|F''(x)\| \leqslant  (2\gamma)/( 1- \gamma \|x-x_0\|)^3.
$$
	\end{lemma}
\begin{proof}
The proof follows the same pattern as the proof of Lemma 21 of \cite{MAX2}.
\end{proof}
\begin{lemma} \label{lc}
Let $\Omega\subseteq\mathbb{R}^n$ be an open set and $F:\Omega\to \mathbb{R}^m$ be twice continuously differentiable on $\Omega$.  If there exists  a \mbox{$f:[0,R)\to \mathbb {R}$} twice continuously differentiable and satisfying
$$
\|F'(x_0)^{\dagger}\|\|F''(x)\|\leqslant f''(\|x-x_0\|),
$$
for all $x\in \Omega$ such that  $\|x-x_0\|<R$, then $F$ and $f$ satisfy \eqref{KH}.
\end{lemma}
\begin{proof}
The proof follows the same pattern as the proof of  Lemma 22 of \cite{MAX2}.
\end{proof}

\noindent
{\bf Proof of Theorem \ref{theo:Smale}.}
Consider the real function   $f:[0,1/\gamma) \to \mathbb{R}$ defined by
$$
f(t)=\frac{t}{1-\gamma t}-2t.
$$
It is straightforward to show that $f$ is analytic and that
$$
f(0)=0, \quad f'(t)=1/(1-\gamma t)^2-2, \quad f'(0)=-1, \quad
f''(t)=(2\gamma)/(1-\gamma t)^3, \quad f^{n}(0)=n!\,\gamma^{n-1},
$$
for $n\geq 2$. It follows from the latter equalities that f satisfies {\bf h1} and
{\bf h2}. Moreover, as $f''(t)=(2\gamma)/(1-\gamma t)^3,$ combining
Lemmas \ref{lemma:qc1} and \ref{lc}, we have $F$ and $f$ satisfy \eqref{KH} with
 $R=1/\gamma.$ Hence,
$$
h_{ \beta, \lambda}(t):= \beta -(2-\lambda)t+ t/(1-\gamma t)=\beta + \lambda  t+ f(t).
$$
Since $\lambda=\kappa$, we have  $0\leq \lambda < 1$ and $\lambda=-\kappa f'(0)\geq -\kappa f'(\beta)$, where the latter inequality follows from {\bf h2}.
Moreover, $\alpha=\beta \gamma\leq3-2\sqrt{2}$ implies that $\big((1-\kappa+\alpha)^2 -4(2-\kappa)\alpha\big)\geq 0$, i.e.,  $h_{ \beta, \lambda}$ satisfies {\bf h3} and   $t_*=\big(1-\kappa+\alpha-\sqrt{(1-\kappa+\alpha)^2 -4(2-\kappa)\alpha}\big)/(2\gamma(2-\kappa))$ is its smallest root.

Therefore, taking $f$, $h_{ \beta, \lambda}$, $\lambda$ and $t_*$  as defined above, all  the statements  of the theorem follow from Theorem~\ref{th:knt}. \qed

Under the Smale's $\alpha$-condition, Theorem~\ref{th:kntmod} becomes:
\begin{theorem}\label{theo:Smale12}
Let $\Omega\subseteq\mathbb{R}^n$ be an open set and $F:\Omega\to \mathbb{R}^m$ an analytic function. Take
 $x_0\in\Omega$ such that $\beta:=\|F'(x_0)^{\dagger}F(x_0)\|>0$ and $F'(x_0)$ is surjective.
Suppose that
\begin{equation} \label{eq:SmaleCond}
  \gamma :=\sup _{ n > 1 }\left\| \frac
{F'(x_0)^{\dagger}F^{(n)}(x_0)}{n !}\right\|^{1/(n-1)}<+\infty, \qquad B(x_0,1/\gamma)\subseteq \Omega.
\end{equation}
Consider the auxiliary function   $h_{ \beta, \kappa}:[0,\; 1/\gamma)\to \mathbb{R}$,
$$
h_{ \beta, \kappa}(t):= \beta -2t+ t/(1-\gamma t).
$$
If $$ \alpha:=\beta \gamma\leq3-2\sqrt{2},$$
then $h_{ \beta, \kappa}(t)$ has a smallest zero   $t_*=\big(1+\alpha-\sqrt{(1+\alpha)^2 -8\alpha}\big)/(4\gamma)$,  the  sequences
for solving  $ h_{ \beta, \kappa}(t)=0$ and $F(x)=~0$, with starting point $t_0=0$ and $x_0$,  respectively,
$$
    t_{k+1} ={t_{k}}-h_{ \beta,0}'(t_{k}) ^{-1}h_{ \beta, \kappa}(t_{k}),\quad x_{k+1}={x_k}-F'(x_k)^{\dagger}F(x_k), \quad k=0,1,\ldots\,,
$$
are well defined,  $\{t_{k}\}$ is strictly increasing, is contained in
  $[0,t_*)$, and converges Q-linearly to $\bar{t}_*$, $\{x_k\}$ is contained in $B(x_0,\bar{t}_*)$ and converges R-linearly to a point  $x_*\in B[x_0, \bar{t}_*]$ such that
$F'(x_*)^{\dagger}F(x_*)=0$,
$$
  \|x_{k+1}-x_{k}\|   \leq  t_{k+1}-t_{k} , \qquad \|x_*-x_{k}\|   \leq  \bar{t}_*-t_{k},\qquad  k=0, 1, \ldots\, ,
$$
$$
\|x_{k+1}-x_{k}\|\leq   \frac{t_{k+1}-t_{k}}{(t_{k}-t_{k-1})^2} \|x_k-x_{k-1}\|^2,
 \quad k=1,2, \ldots\,,
$$
$$
  \|F'(x_0)^{\dagger}F(x_k)\|\leq \left(  \frac{t_{k+1}-t_{k}}{t_{k}-t_{k-1}}\right)\|F'(x_0)^{\dagger}F(x_{k-1})\| ,\quad k=1,2, \ldots\,.
$$
If, additionally, $ \alpha:=\beta \gamma<3-2\sqrt{2}$, then the sequences $\{t_{k}\}$ and $\{x_{k}\}$ converge Q-quadratically and R-quadratically to $t_{*}$ and  $x_{*}$, respectively.
\end{theorem}
\begin{proof}
The proof follows the same pattern as the proof of  Theorem~\ref{theo:Smale}.
\end{proof}

%%%%%%%%%%%%%%%%%%%%
\section{Final remarks}

We  presented a  new semi-local convergence analysis of the Gauss-Newton method for solving \eqref{eq:11}, where $F$ satisfies \eqref{rank1}, under a majorant condition. It would also be interesting to present a local convergence analysis of the Gauss-Newton method, under a majorant condition, for the problem under consideration. As a consequence, we would get convergence results for analytical functions under an $\gamma$-condition. This local analysis will be performed in the future.

%%%%%%%%%%%%%%%%%%%%%%%%%%%%%%%%%%%%%%%%%%%%%%%%%%%%%%%%%%%%%%%%%%%%%%%%%%%%%%%%
%\begin{thebibliography}{99}
%\bibliographystyle{alpha}
%\bibliographystyle{plain}
%\bibliographystyle{siam}
%\bibliographystyle{unsrt}
%\bibliographystyle{abbrv}%inexact gauss-newton and comp conv
%\bibliography{c:/orizon/Documentos/papers/bibtex/newton}
%\bibliography{RefGaussNewton}
%\end{thebibliography}
%\end{document}
%%%%%%%%%%%%%%%%%%%%%%%%%%%%
%%%%%%%%%%%%%%%%%%%%%%%%%%%
\def\cprime{$'$}

\end{document}